\documentclass[11pt]{amsart}

\usepackage{amsmath,amssymb}
\usepackage{color}
\usepackage[all,curve,frame]{xy}

\usepackage[pdftex]{graphicx}

\newtheorem{defi}{Definition}[section]
\newtheorem{sinobservacion}[defi]{}
\newenvironment{sinob}{\begin{sinobservacion} \rm}{\end{sinobservacion} }
\newtheorem{coro}[defi]{Corollary}
\newtheorem{lema}[defi]{Lemma}
\newtheorem{nota}[defi]{Notation}
\newtheorem{obs}[defi]{Remark}
\newtheorem{prop}[defi]{Proposition}
\newtheorem{teo}[defi]{Theorem}
\newtheorem{ej}[defi]{Example}

\newcommand{\benu}{\begin{enumerate}}
\newcommand{\enu}{\end{enumerate}}

\newcommand{\al}{\alpha}

\newcommand{\mbmA}{\mbox{mod}\,A}

\newcommand{\fle}{\rightarrow}

 \usepackage{xcolor} 
\definecolor{darkgreen}{RGB}{55,138,0}
  \definecolor{burntorange}{RGB}{180,85,0}
\definecolor{navyblue}{RGB}{18,40,180}
\definecolor{cyan(process)}{rgb}{0.0, 0.6, 1.0}
\definecolor{blue-violet}{rgb}{0.54, 0.17, 0.89}

\usepackage{xcolor}

\usepackage[color=blue!20!white]{todonotes} 
\usepackage{hyperref}

\begin{document}
\title[A generalization of the nilpotency index]
{A generalization of the nilpotency index of the radical of the module category of an algebra}
\author[Chaio]{Claudia Chaio}
\address{Centro marplatense de Investigaciones Matem\'aticas, Facultad de Ciencias Exactas y
Naturales, Funes 3350, Universidad Nacional de Mar del Plata, 7600 Mar del
Plata, CONICET Argentina}
\email{cchaio@mdp.edu.ar}

\author[Suarez]{Pamela Suarez}
\address{Centro marplatense de Investigaciones Matem\'aticas, Facultad de Ciencias Exactas y
Naturales, Funes 3350, Universidad Nacional de Mar del Plata, 7600 Mar del
Plata, Argentina}
\email{psuarez@mdp.edu.ar}

\maketitle

\begin{abstract}
Let $A$ be a finite dimensional representation-finite algebra over
an algebraically closed field. The aim of this  work is to generalize the results proven in \cite{CGS}. Precisely, we determine which vertices of $Q_A$ are sufficient  to be considered in order to compute the nilpotency index of the radical of the module category of a monomial algebra and a toupie algebra $A$, when the Auslander-Reiten quiver is not necessarily a component with length.
\end{abstract}

\section*{Introduction}
Let $A$ be a finite dimensional $k$-algebra over an algebraically closed field $k$, and  $\mbox{mod}\,A$ the category of finitely generated right $A$-modules.


  For an indecomposable $A$-module $X$ we know that the radical of $\mbox{End}_A\,X$, denoted by $\Re(\mbox{End}_A\,X)$,
 is the set of all non-isomorphisms from $X$  to $X$. This concept was extended  to  any  $A$-modules $X$ and $Y$ not necesarily indecomposable, defining the ideal  $\Re(\mbox{Hom}_A(X,Y))$ of $\mbox{mod}\,A$ as the set of all morphisms $f:X \rightarrow Y$ such that for every indecomposable $A$-module $M$ and every morphism $h: M \rightarrow X$ and $h': Y\rightarrow M$ we have that $h'fh \in \Re(\mbox{End}_A\,M)$. Inductively, the powers of $\Re(\mbox{Hom}_A(X,Y))$ are defined and  $\Re^{\infty}(\mbox{Hom}_A(X,Y))= \cap_{n\in \mathbb{N}} \Re^{n}(\mbox{Hom}_A(X,Y))$.

 In case that we deal with a representation-finite algebra, it is well-known that  $\Re^{\infty}(\mbox{mod}\,A)$ vanishes, that is,  $\Re^{\infty}(\mbox{Hom}_A(X,Y))=0$ for all $A$-modules $X$ and $Y$. Therefore there is a positive integer $n$ such that $\Re^{n}(\mbox{mod}\,A)=0$, see \cite[p. 183]{ARS}.
The minimal lower bound $m \geq 1$ such that $\Re^{m}(\mbox{mod}\,A)$ vanishes
is called the nilpotency index of $\Re(\mbmA)$.


It is  known that  the nilpotency index of the radical of $\mbox{mod}\, A$ for $A$ a representation-finite algebra, is a bound that coincides with the length of the longest non-zero path from the projective corresponding to  a vertex $a$ of $Q_A$ to the injective corresponding to the same vertex  going through the simple in $a$.

In \cite{C},  the first named author studied the problem to determine the nilpotency index of the radical of $\mbox{mod}\, A$ for $A$ a representation-finite algebra in terms of the left and right degrees of some particular irreducible morphisms. Precisely, for $A \simeq kQ_A/I_A$ it is enough to study the left degree of the irreducible morphisms from $I_a$ to $I_a/S_a$ and the right degree of the irreducible morphisms from the radical of $P_a$ to $P_a$, for  all the vertices $a$ in $Q_A$ where $P_a$, $I_a$ and $S_a$ are the projective, the injective and the simple, respectively,  corresponding to a vertex $a$ in $Q_A$. The concept of degree of an irreducible morphism  allow us, in many cases,  to determine the nilpotency bound without constructing  the Auslander-Reiten quiver of  $\mbox{mod}\, A$. This  notion  was introduced in 1982 by S. Liu, see \cite{L}. Since 2004 there has been a development on the degree theory and  the concept of degree showed to be a powerful tool to solve many problems about the radical of a module category.

In 2013, S. Liu and the first named author studied the problem to determine the nilpotency index of $\Re(\mbox{mod}\, A)$ for any representation-finite artin algebra $A$, see  \cite{CL}.

Continuing with this kind of problems, in \cite{CG2}, C. Chaio and V. Guazzelli found the first result concerning how to reduce the steps in order to compute the nilpotency index of the radical of a module category of an algebra of representation-finite type. Although to compute it requires to analyze a finite number of particular paths, it is of interest  to reduce the process. Precisely, in such a work it was proved that if $A$ is a finite dimensional algebra over an algebraically closed field, then  it is not necessary to analyze the length of the paths from the projective to the injective  going through the simple corresponding to a vertex $a$  such that $a$ is a  sink or a source  of $Q_A$.

In \cite{CGS},  the two named authors and V. Guazzelli, continued analyzing how to reduce the steps to find the nilpotency bound, studying which vertices of $Q_A$ are sufficient to be consider in order to determine the nilpotency index of the radical of the module category of an algebra. The mentioned  authors studied the case of representation-finite monomial and toupie algebras where their Auslander-Reiten quivers are components with length. By a component with length we mean a component where  parallel paths have  the same length. Precisely, in case of representation-finite  monomial algebras they found that   we only have to study the length of the paths of irreducible morphisms between indecomposable modules from the projective
$P_u$ to the injective $I_u$, where $u$ belongs to the set of vertices involved in zero-relations.  By a vertex $u$ in $Q_A$  involved in  a zero-relation of an ideal $I_A$, we mean that if  $\alpha_m \dots \alpha_1 \in I_A$  then  $u =s(\alpha_i)$ for some $i=2, \dots, m$.

On the other hand, for toupie algebras  the authors studied the case where $I_A$ has only  commutative relations. They proved that  any vertex of $Q_A$ which is in the shortest  branch of $Q_A$, not being the sink or the source of $Q_A$,  is the one that  we may consider to determine the nilpotency index of the radical of $\mbox{mod}\,A$.

Moreover, the above mentioned authors also proved that for a representation-finite  string algebra it is sufficient to study the vertices $u \in Q_A$ that are involved in zero-relations of $I_A$ in order to determine the nilpotency index of the radical of $\mbox{mod}\,A$. They also found such a bound, by using results proven in \cite{CG} of  how to read the left and right degrees of irreducible morphisms between indecomposable $A$-modules in any string algebra.  We observe that representation-finite string algebras may have an Auslander-Reiten component without length.
\vspace{.05in}

The aim of this work is to generalize the results proven in \cite{CGS} to any representation-finite algebra. Precisely, we prove Theorem A.
\vspace{.1in}

{\bf Theorem A ([Corollary \ref{irred}])}. {\it
Let $A \simeq kQ_A/I_A$ be a finite dimensional $k$-algebra over an algebraically closed field.  Assume that $A$ is a representation-finite algebra. Let $a$ and $b \in (Q_A)_0$ such that there is an arrow from $a$ to $b$ in $Q_{A}$. Let  $r_a$ be the length of the non-zero path of irreducible morphisms from the projective $P_a$ to the injective $I_a$ going through the simple $S_a$. Then the following statements hold.
\begin{enumerate}
\item[(a)] If there is an irreducible morphism from $P_b$ to $P_a$ such that   $\emph{End} (P_b)\simeq k$ then $r_b \leq r_a$.
\item[(b)] If there is an irreducible morphism from $I_b$ to $I_a$  such that  $\emph{End} (I_a)\simeq k$ then $r_a \leq r_b$.
\end{enumerate}}
\vspace{.1in}

As an application of Theorem A, we obtain results concerning monomial and toupie algebras where their Auslander-Reiten quiver may be components without length.  More precisely, we prove Theorem B, Theorem C  and Theorem D.
\vspace{.1in}

{\bf Theorem B ([Theorem \ref{vertex-nilpo}])}. {\it
Let $A\simeq kQ_A/I_A$ be a representation-finite monomial algebra. Consider $(R_A)_0$ to be the set of  the vertices $u\in Q_A$
involved in zero-relations of $I_A$. Then the nilpotency index  $r_A$ of $\Re(\emph{mod}\,A)$ is determined by the length $r_u$ of the non-zero  path of irreducible morphisms from the projective $P_u$ to the injective $I_u$ going through the simple $S_u$, where $u \in(R_A)_0$. Precisely,  $r_A = \emph{max}\, \{ r_{a}+1 \}_{a \in (R_A)_0}$. }
\vspace{.1in}

{\bf Theorem C ([Theorem \ref{rarb}])}. {\it
Let $A\simeq kQ_A/I_A$ be a representation-finite monomial algebra. Consider $(R_A)_0$ to be the set of the vertices $u\in Q_A$
involved in zero-relations of $I_A$.  If each vertex in $(R_A)_0$ belongs to only one zero-relation only once then the nilpotency index of $\Re(\emph{mod}\,A)$ is determined by the length of the  non-zero  path of irreducible morphisms from the projective $P_u$ to the injective $I_u$, where $u \in(R_A)_0$. Furthermore, all the vertices $x$ related to a fixed zero-relation determine paths from $P_x$ to $I_x$ going through $S_x$ of the same length. Therefore,  it is enough to consider only one vertex in each zero-relation. }
\vspace{.1in}

{\bf Theorem D([Theorem \ref{toupie}])}. {\it Let $A \simeq kQ_A/I_A$ be a  representation-finite toupie algebra where $I_A$ has three branches where a branch has only one zero-relation and the others have  commutative relations. Then the nilpotency index of $\Re(\emph{mod}\,A)$ is determined by the length of the non-zero path of irreducible morphisms from the projective $P_u$ to the injective $I_u$ going through the simple $S_u$, where $u$ is a vertex involved in a zero-relation. More precisely, it is enough to study only one vertex in the zero-relation of $I_A$. }
\vspace{.1in}

We also show an example where Theorem B is not true when the  branch having the zero-relation has more than one zero-relation. In fact, in such a case it depends on the length of the branches.
\vspace{.05in}

This paper is organized as follows. In the first section, we recall some preliminaries results.
Section 2 is dedicated to prove a general result to reduce the number of paths that we have to analyze in order to determine the nilpotecy index of $\Re(\mbox{mod}\,A)$, for any finite-dimensional $k$-algebra of representation-finite type and we prove Theorem A.
In Section 3, we generalize the result proven in \cite{CGS} about the nilpotency index of the radical of a module category of a representation-finite monomial algebra where the Auslander-Reiten quiver is a component with length to any  representation-finite monomial algebra.  We prove Theorem B and Theorem C. In Section 4, we study the nilpotency index of the radical of a module category of a representation-finite having only one zero-relation. Note that in such a case the  Auslander-Reiten quiver is a component without length. Precisely, we prove Theorem D.
\vspace{.1in}

\thanks{The authors thankfully acknowledge partial support from Universidad Nacional de Mar del Plata, Argentina. Both authors want to thank Victoria Guazzelli for helpful conversations. The first author is a researcher from CONICET.}

\section{preliminaries}

Throughout this work, by an algebra we mean a finite dimensional $k$-algebra over an algebraically closed field,  $k$.

\begin{sinob} 

A {\it quiver} $Q$ is given by a set of vertices $Q_0$ and
a set of arrows $Q_1$, together with two
maps $s,e:Q_1\fle Q_0$.  Given $\al\in Q_1$,  we denote by $s(\al)$ and 
 by $e(\al)$ 
the starting and
the ending vertex of the arrow $\al$, respectively.


Let $A$ be an algebra.  Then by Gabriel's theorem we know that there exists a quiver $Q_A$, called the {\it ordinary quiver of}  $A$, such that
$A$ is the quotient of the path algebra $kQ_A$ by an admissible ideal.

We denote by $\mbox{mod}\,A$ the category of finitely generated
right $A$-modules and  by  $\mbox{ind}\,A$ the full subcategory of $\mbox{mod}\,A$ which consists of
one representative of each isomorphism class of indecomposable $A$-modules.

 We say that an algebra A is {\it representation-finite} if there is only a finite number
of isomorphisms classes of indecomposable A-modules.

\begin{lema} \label{multiplicity}
Let $A$ be a finite dimensional $k$-algebra. Consider $M$   an indecomposable  $A$-module, $S$  a simple $A$-module and $I$  the injective hull
of~$S$. Then $\dim_k (\emph{Hom}_{A}(M,I))=[M:S]$, the multiplicity of~$S$ in a composition
series for~$M$.
\end{lema}

A similar result holds for  the projective cover of~$S$.
\vspace{.1in}


The Auslander-Reiten quiver of $\mbox{mod}\, A$, denoted by $\Gamma_A$, is a translation quiver with vertices the classes of isomorphisms of indecomposable A-modules. We denote by $[M]$ the vertex corresponding to an indecomposable $A$-module $M$. There is an arrow $[M]\rightarrow [N]$ between two vertices in $\Gamma_A$ if and only if there is an irreducible
morphism from $M$ to $N$ in $\mbox{mod}\,A$.

From now on, we do not distinguish between the indecomposable $A$-modules and the corresponding vertices of $\Gamma_A$.

For a finite dimensional algebra over an algebraically closed field, it is known that an arrow  $[M] \rightarrow [N]$ of $\Gamma_A$ has valuation $(a,a)$, that is, there is a minimal right almost split
morphism $aM \oplus X \rightarrow N$ where $M, N$ are indecomposable $A$-modules and $M$ it is not a summand of $X$, and there
is a minimal left almost split morphism $M  \rightarrow aN\oplus Y$ where $M, N$ are indecomposable  $A$-modules and $N$ is not
a summand of $Y$. We say that the arrow has trivial valuation if $a = 1$.  In particular, if $A$ is a representation-finite algebra then the valuation of the arrows of $\Gamma_A$ are trivial.

We say that a component $\Gamma$ of $\Gamma_A$ is standard if the full subcategory of $\mbox{mod}\,A$ generated by the
modules of $\Gamma$  is equivalent to the mesh category $k(\Gamma)$ of $\Gamma$, see \cite{BG}. The mesh-category of $\Gamma$, $k(\Gamma)$, is the quotient category
$k\Gamma/ I$ where $I$ is the mesh ideal of $\Gamma_A$.

For a finite dimensional algebra over an algebraically closed field, it is known
that  if $A$ is  representation-finite  then $\Gamma_A$ is standard.

 Following \cite{CPT},  we say that a component $\Gamma$ of $\Gamma_A$ is with length if parallel paths in $\Gamma$ has the same length. Otherwise, we say that $\Gamma$ is a component without length. We observe that a component $\Gamma$ of $\Gamma_A$ without length may have cycles. 
\vspace{.1in}

For unexplained notions in representation theory of algebras, we refer the reader to \cite{ASS} and \cite{ARS}.
\end{sinob}

\begin{sinob} \label{rad}

Consider  $X,Y \in \mbox{mod}\,A$. The {\it Jacobson radical} of $\mbox{Hom}_A(X,Y)$, denoted by   $\Re(X,Y)$,  is
the set of all the morphisms $f: X \rightarrow Y$ such that, for each
$M \in \mbox{ind}\,A$, each $h:M \rightarrow X$ and each $h^{\prime }:Y\rightarrow M$
the composition $h^{\prime }fh$ is not an isomorphism. For $n \geq 2$, the powers of the radical
are inductively  defined  as $\Re^{n}(X,Y)= \Re^{n-1}(M,Y)\Re(X,M) = \Re (M,Y)\Re^{n-1}(X,M)$ for some $M \in \mbox{mod}\,A$.

A morphism $f :X  \rightarrow  Y$, with $X,Y \in \mbox{ mod}\,A$,
is called {\it irreducible} provided it does not split
and whenever $f = gh$, then either $h$ is a split monomorphism or $g$ is a
split epimorphism. If $X, Y$ are indecomposable $A$-modules, then $$\mbox{Irr}(X, Y) = \Re (X, Y)/\Re^2(X, Y)$$ is a
$k$-bimodule. 

As a consequence, if $f : X \rightarrow Y$ is an irreducible  morphism between indecomposable $A$-modules and $\mbox{dim}_{k} \mbox{Irr}(X, Y)=1$ then any other irreducible morphism $g: X \rightarrow Y$ in $\mbox{mod}\, A$ can be written as $g= \alpha f+ \mu$, where $\alpha \in k-\{0\}$ and $\mu \in \Re^2(X, Y)$.
\end{sinob}

\begin{sinob}

It is well-known that an algebra $A$ is representation-finite
if and only if there is a positive integer $n$ such that $\Re^n(X,Y)=0$ for each $A$-module $X$ and $Y$, see \cite[V, Theorem 7.7]{ARS}.
The minimal lower bound $m$ such that $\Re^m(\mbox{mod}\,A)=0$ is called the \textit{nilpotency index} of
$\Re(\mbox{mod}\,A)$. We denote it by $r_A$.

We denote by $P_a$, $I_a$ and  $S_a$ the projective, the injective and the simple $A$-module corresponding to the
vertex $a$ in $Q_A$, respectively.

\begin{defi} By a non-zero path of irreducible morphisms, we mean a non-zero composition of irreducible morphisms between indecomposable modules in $\emph{mod}\,A$.
\end{defi}

\begin{defi} \label{largo}
Let $A$ be a finite dimensional $k$-algebra over an algebraically closed field. We say that a non-zero path $\varphi : X \rightarrow Y$ of irreducible morphisms between indecomposable $A$-modules  in $\mbox{mod}\,A$ is of the length $n$ if $\varphi \in \Re^{n}(X,Y)\backslash \Re^{n+1}(X,Y)$.
 \end{defi}

As an immediate consequence of \cite[Proposition 3.11]{CG2} and its dual,
to compute the nilpotency index of the radical of a module
category of a representation-finite algebra, it is enough to analyze
the vertices which are neither sinks nor sources in $Q_A$, as we state below.

\begin{teo}\label{nilpopf} \cite[Theorem 3.16]{CG2}
Let $A=kQ_A/I_A$ be a representation-finite algebra.
Consider $\mathcal{V}$ to be the set of vertices of $Q_A$ which are neither sinks nor sources.
Suppose that $\mathcal{V}\neq \emptyset$. Then the  nilpotency index of $\Re(\emph{mod}\,A)$ is
$\emph{max}\,\{r_a+1\}_{a \in \mathcal{V}}$, where $r_a$ is the  length of the non-zero path $P_a \rightsquigarrow S_a \rightsquigarrow I_a$ with $a \in \mathcal{V}$.
\end{teo}
\end{sinob}

\begin{sinob}

We recall the following  lemmas proved in \cite{C}, which shall be useful throughout this work.

\begin{lema} \cite[Lemma 2.1]{C}  \label{composite} Let $A$ be a finite dimensional algebra over an
algebraically closed field $k$. Let $f\in
\Re^n(P_a,S_a)\backslash\Re^{n+1}(P_a,S_a)$ and $g \in
\Re^m(S_a,I_a)\backslash\Re^{m+1}(S_a,I_a)$ for some vertex $a \in Q_0$. Then $gf \in
\Re^{n+m}(P_a,I_a)\backslash \Re^{n+m+1}(P_a,I_a).$
\end{lema}

\begin{lema} \cite[Lemma 2.2]{C} \label{refe} Let $A \simeq kQ/I$ be a finite dimensional algebra over an
algebraically closed field of finite representation type. Let $f:P_a \longrightarrow I_b$ be a non-zero morphism between the
indecomposable projective and injective $A$-modules corresponding
to some vertices $a,b \in Q_0,$ respectively. If $f$ does not factor
through the simple $S_a,$ then there exists a non isomorphism
$\varphi: I_b \longrightarrow I_a$ such that $\varphi f$ is
non-zero and factors through $S_a$.
\end{lema}
We observe that there is also a dual result of Lemma \ref{refe}. We state it below.
\begin{lema} \label{refe-dual} Let $A \simeq kQ/I$ be a finite dimensional algebra over an
algebraically closed field of finite representation type. Let $f:P_a \longrightarrow I_b$ be a non-zero morphism between the
indecomposable  projective and injective $A$-modules corresponding
to some vertices $a,b \in Q_0,$ respectively. If $f$ does not factor
through the simple $S_b,$ then there exists a non isomorphism
$\varphi: P_b \longrightarrow P_a$ such that $\varphi f$ is
non-zero and factors through $S_b$.
\end{lema}

\begin{lema} \cite[Lemma 2.4]{C}  \label{proj-inj} Let $A \simeq kQ/I$ be a finite dimensional algebra over an algebraically closed field of finite representation type and $a \in (Q_A)_0$. Let  $r_a$ be the length of the  non-zero path $P_a \rightsquigarrow S_a \rightsquigarrow I_a$. Then the following statements hold.
\begin{enumerate}
\item[(a)] Every  non-zero morphism $f: P_a \longrightarrow I_a$ that factors
through the simple $A$-module $S_a$ is such that $f \in \Re^{r_a}(P_a,I_a)\backslash \Re^{r_a +1}(P_a,I_a).$ 

\item[(b)] Every non-zero morphism $f: P_a \longrightarrow I_a$ which  does not factor
through the simple $A$-module $S_a$ is such that $f \in
\Re^{k}(P_a,I_a)\backslash \Re^{k+1}(P_a,I_a)$ with $0 \leq k < r_a$.
\end{enumerate}
\end{lema}


The next theorem is fundamental for the proofs of our results.

\begin{teo} \cite[Theorem 2.7]{CT} \label{explotan}  Let $\Gamma$ be a standard component of $\Gamma_A$ with trivial valuation and $X_i \in \Gamma$, for $i=1, \dots,  n + 1$. The following conditions are equivalent:
\begin{enumerate}
\item[(a)] There are irreducible morphisms $h_i : X_i \rightarrow X_{i+1}$ such that $h_n \dots h_1 \neq 0$ and $h_n \dots  h_1 \in \Re^{n+1}$.
\item[(b)]  There is a path $X_1 \stackrel{f_1}\rightarrow X_2 \stackrel{f_2}\rightarrow X_3 \rightarrow \dots \rightarrow X_n \stackrel{f_n}\rightarrow X_{n+1}$ of irreducible morphisms with zero composition and a morphism $\epsilon= \epsilon_n ... \epsilon_1 \neq 0$ where $\epsilon_i = f_i$  or $\epsilon_i \in \Re^{2}(X_i, X_{i+1})$.
\end{enumerate}
\end{teo}

For further information in a more general result than Theorem \ref{explotan}, we refer the reader to \cite{CLT}.
\end{sinob}


\section{General results}


Now, we establish some notation that we shall use throughout this paper.

\begin{nota} \label{notacion} Let $A \simeq kQ_A/I_A$ be a finite dimensional $k$-algebra over an algebraically closed field. Let $P_a$, $I_a$, and $S_a$ be the projective, the injective, and the simple corresponding to the vertex $a$ in $Q_A$. We denote by $\varphi_{_a}$ the non-zero path of irreducible morphisms from $P_a$ to $I_a$ going through the simple $S_a$ and by $r_a$ the length of  $\varphi_{_a}$, whenever $A$ is a representation-finite algebra.
\end{nota}





Next, we state the main result of this section.

\begin{teo}\label{teoirred}
Let $A \simeq kQ_A/I_A$ be a finite dimensional $k$-algebra over an algebraically closed field.  Assume that $A$ is a representation-finite algebra. Let $a$ and $b$ be vertices in  $Q_A$ such that there is an arrow from $a$ to $b$ in $Q_{A}$.
\begin{enumerate}
\item[(a)] If there is an irreducible morphism from $P_b$ to $P_a$ and any non-zero morphism $\eta: P_b \rightarrow P_a$ is such that $\eta= \mu f_1$ where $\mu: P_a \rightsquigarrow P_a$ is a non-zero morphism and $ f_1$ is an irreducible morphism from $P_b$ to $P_a$  then $r_b \leq r_a$.
\item[(b)] If there is an irreducible morphism from $I_b$ to $I_a$  and  any non-zero morphism $\eta: I_b \rightarrow I_a$ is such that $\eta= g_1\mu$ where $\mu: I_b \rightsquigarrow I_b$ is a non-zero morphism  and $g_1$ is an irreducible morphism from $I_b$ to $I_a$ then $r_a \leq r_b$.
\end{enumerate}
\end{teo}
\begin{proof}
(a) Let $\phi\in \mbox{Hom}_A(P_a,I_b)$ be a non-zero path of irreducible morphisms between indecomposable $A$-modules of maximal length. Assume that $\ell(\phi)=n$, that is, $\phi\in \Re^n(P_a,I_b)\backslash \Re^{n+1}(P_a,I_b)$. Since $\phi$ does not factor through the simple $S_b$ by Lemma \ref{refe} there exists a non-isomorphism $\varphi: P_b\rightarrow P_a$ such that $\phi\varphi$ is non-zero and factors through $S_b$. 

We claim that $\varphi$ is an irreducible morphism.  In fact, by hyphotesis $\varphi=\mu f$ with $f$ an irreducible morphism from $P_b$ to $P_a$ and $\mu$ a non-zero morphism from $P_a$ to $P_a$. If $\mu$ is non-trivial,  then
we get a non-zero morphism $\phi\mu$ from $P_a$ to $I_b$. Moreover  since $A$ is representation-finite, then  there exists a non-zero path of irreducible morphism between indecomposable $A$-modules $\phi\mu':P_a\rightarrow I_b\in \Re^{n+1}(P_a, I_b)$ which is a contradiction to the maximality of the length of $\phi$.  Therefore, $\varphi: P_b\rightarrow P_a$ is an irreducible morphism. 

Thus we obtain a non-zero path  $\phi \varphi:P_b \rightarrow I_b$ that factors through $S_b$, where $\varphi:P_b\rightarrow P_a$ is an irreducible morphism. By Lemma \ref{proj-inj} we have that $\phi \varphi\in \Re^{r_b}(P_b,I_b)\backslash \Re^{r_b+1}(P_b,I_b)$.  

We affirm that $\phi\in \Re^{r_b-1}(P_a,I_b) $.    Indeed, assume that  $n<r_b-1$ and consider \[\phi:X_1 \stackrel{h_1} \rightarrow  X_2  \stackrel{h_2}  \rightarrow  \dots \stackrel{h_{n-1}} \rightarrow X_{n-1} \stackrel{h_{n}} \rightarrow  X_{n+1}\]  where $X_1=P_a$ and $X_{n+1}=I_b$. Then by  Theorem \ref{explotan} we know that there is a  path $P_b=X_0\stackrel{f_0}\rightarrow P_a=X_1 \stackrel{f_1}\rightarrow X_2 \stackrel{f_2}\rightarrow X_3 \rightarrow \dots \rightarrow X_{n-1} \stackrel{f_n}\rightarrow X_{n+1}=I_b$ of irreducible morphisms with zero composition and a morphism $\epsilon= \epsilon_{_{n}} ... \epsilon_{_0} \neq 0$ where $\epsilon_{_i} = f_{_i}$  or $\epsilon{_i} \in \Re^{2}(X_i, X_{i+1})$. Observe that at least for one $i$ we have that $\epsilon{_i} \in \Re^{2}(X_i, X_{i+1})$, otherwise $\epsilon=0$. 

If for any $1\leq i \leq n$ we get that $\epsilon_i\in \Re^2(X_i,X_{i+1})$ then we obtain a path in $\Re^{n+1}(P_a,I_b)$ which is a contradiction to the maximality of $\phi$. Then $\epsilon_0\in \Re^{2}(P_a,P_b)$. By hypothesis, $\epsilon_0=\nu g$ where $\nu:P_a\rightarrow P_a$ is a non-zero morphism and $g$ is an irreducible morphism from $P_b$ to $P_a$. 

Again we obtain a non-zero morphism $\epsilon_{_{n}} ... \epsilon_{_{1}}\nu\in \Re^{n+1}(P_a, I_b)$ which contradicts the maximality of $\phi$. 
Hence $\phi\in \Re^{r_{b}-1}(P_a, I_b) \backslash \Re^{r_{b}}(P_a, I_b)$. 

By Lemma \ref{refe}, since  $\phi$ does not factor
through the simple $S_a$,  then there exists a non isomorphism $\phi': I_b \longrightarrow I_a$ such that $\phi' \phi$ is
non-zero and factors through $S_a$.  We illustrate the situation as follows:  
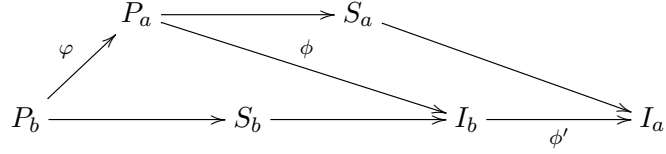
\begin{displaymath}
\xymatrix { & P_a \ar[rrrd]^{\phi}\ar[rr] && S_a \ar[rrrd] & & &\\
P_b  \ar[ru]^{\varphi} \ar[rr] && S_b \ar[rr]&& I_b  \ar[rr]_{\phi'}& & I_a \\}
 \end{displaymath}
Since $\phi'$ is a non-zero morphism we get that $\phi' \phi\in \Re^{r_b}(P_a, I_a)$. On the other hand, since $\phi' \phi$ factors through $S_a$ by Lemma \ref{proj-inj} we have that $\phi' \phi\in \Re^{r_a}(P_a, I_a)\backslash \Re^{r_a+1}(P_a, I_a)$. Therefore, $r_b \leq r_a$ proving Statement (a).

\vspace{.1in}

(b) The proof of (b) follows with dual arguments than the ones used in Statement (a). 
\end{proof}

 Now, we show an example where applying Theorem \ref{teoirred} we determine the nilpotency index of the radical of the module category of an algebra which is not string and such that the Auslander-Reiten quiver is a component without length. 

\begin{ej}
   \emph{ Consider $A=kQ_A/I_A$ the algebra given by the quiver 
   
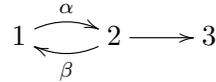
\[\xymatrix{1\ar@/^/[r]^{\alpha} &
2\ar@/^/[l]^{\beta}\ar[r]& 3} \] 
with $I_A=<\alpha\beta\alpha>$. The Auslander-Reiten quiver of $\mbox{mod}\,A$ is as follows: 
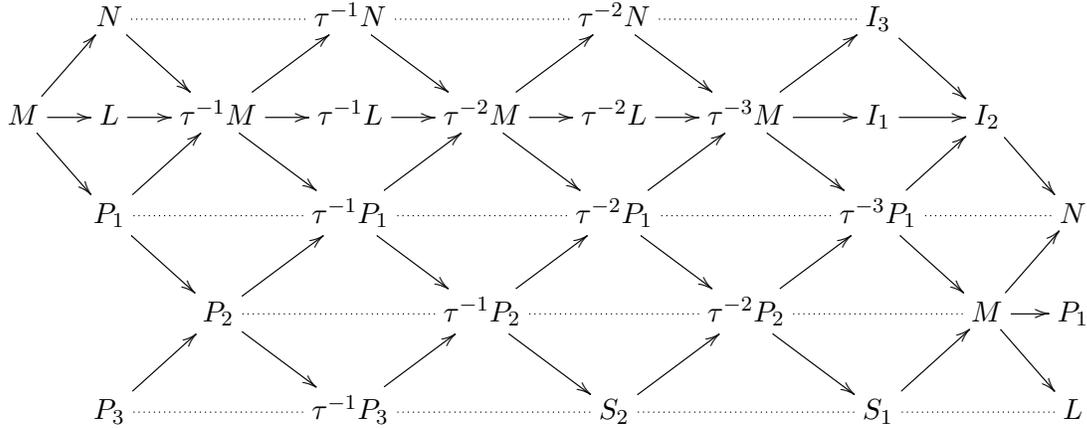
\[\xymatrix @R=0.7cm  @C=0.5cm{
&N\ar[dr]\ar@{.}[rr]&&\tau^{-1}N\ar[dr]\ar@{.}[rr]&&\tau^{-2}N\ar[dr]\ar@{.}[rr]&&I_3\ar[dr]&&\\
M\ar[dr]\ar[ur]\ar[r]& L\ar[r] &\tau^{-1}M\ar[dr]\ar[ur]\ar[r]& \tau^{-1}L\ar[r] & \tau^{-2}M\ar[dr]\ar[ur]\ar[r]&\tau^{-2}L\ar[r]& \tau^{-3}M\ar[dr]\ar[ur]\ar[r]&I_1\ar[r]& I_2\ar[dr]&\\
&P_1\ar[dr]\ar[ur]\ar@{.}[rr]&& \tau^{-1}P_1\ar[dr]\ar[ur]\ar@{.}[rr]&&\tau^{-2}P_1\ar[dr]\ar[ur]\ar@{.}[rr]&& \tau^{-3}P_1\ar[dr]\ar[ur]\ar@{.}[rr]&& N\\
&&P_2\ar[dr]\ar[ur]\ar@{.}[rr]&&\tau^{-1} P_2\ar[dr]\ar[ur]\ar@{.}[rr]&& \tau^{-2} P_2\ar[dr]\ar[ur]\ar@{.}[rr]&& M \ar[ur] \ar[r]\ar[dr]& P_1\\
&P_3\ar[ur]\ar@{.}[rr]&& \tau^{-1}P_3\ar[ur]\ar@{.}[rr]&&S_2\ar[ur]\ar@{.}[rr]&& S_1\ar[ur]\ar@{.}[rr]&& L}\]}
\vspace{.1in}

\noindent \emph{where we identify the $A$-modules $M, N, L$ and $P_1$.}

\emph{We observe that every non-zero morphism $\nu:P_1\rightarrow P_2$ factors through an irreducible morphism from $P_1$ to $P_2$. Then applying Theorem \ref{teoirred} we get that $r_1\leq r_2$.}

\emph{Similarly, since every non-zero morphism $\nu':I_1\rightarrow I_2 $ factors  through an irreducible morphism from $I_1$ to $I_2$ applying again  Theorem \ref{teoirred} we get that $r_2\leq r_1$. Hence $r_1=r_2$.} 

\emph{Moreover, in this case $r_1=r_2=14$. By Theorem \ref{nilpopf} we get that the nilpotency index of the module category of $A$ is equal to fifteen. }

\emph{Finally, we observe that $\mbox{dim}_k\mbox{End}_A(P_1)=2$ and $\mbox{dim}_k\mbox{End}_A(I_2)=2$.}
\end{ej}

As a consequence of Theorem \ref{teoirred} we obtain the following useful corollary.

\begin{coro}\label{irred}
Let $A \simeq kQ_A/I_A$ be a finite-dimensional $k$-algebra over an algebraically closed field.  Assume that $A$ is a representation-finite algebra. Let $a$ and $b \in (Q_A)_0$ such that there is an arrow from $a$ to $b$ in $Q_{A}$.
\begin{enumerate}
\item[(a)] If there is an irreducible morphism from $P_b$ to $P_a$ such that   $\emph{End} (P_b)\simeq k$ then $r_b \leq r_a$.
\item[(b)] If there is an irreducible morphism from $I_b$ to $I_a$  such that  $\emph{End} (I_a)\simeq k$ then $r_a \leq r_b$.
\end{enumerate}
\end{coro}
\begin{proof}
We only state the proof of Statement (a) since Statement (b) follows with dual arguments.

Let $\varphi\in \mbox{Hom}_A(P_b, P_a)$. Since $\mbox{Hom}_A(P_b, P_a)$ has a basis consisting of all paths from $a$ to $b$ in $Q_A$ modulo the ideal $I_A$ and  $\mbox{End}_{A}(P_b) \simeq k$ we have that there is a cycle from $a$ to $a$ in $Q_A$. Moreover, for each path $\delta$ from $a$ to $b$ we have a morphism $g_{_{\delta}}:P_b\rightarrow P_a$ given by $g_{_{\delta}}(\lambda)=\lambda\delta$, where $\lambda$ is a path starting in $b$. In particular, the arrow $\alpha$ from $a$ to $b$ gives rise to an irreducible monomorphism $f$  from $P_b$ to $P_a$. Since any other path from $a$ to $b$ in $Q_A$ is of the form $ \alpha\rho$, with $\rho: a  \rightsquigarrow a$, we obtain that any morphism from $P_b$ to $P_a$ factors thought the irreducible morphism $f$. Therefore, for any non-zero morphism
$\varphi:P_b\rightarrow P_a$ we get that $\varphi=\mu f$ with $f$ an irreducible morphism from $P_b$ to $P_a$ and $\mu$ a non-zero morphism from $P_a$ to $P_a$. The result follows from Theorem \ref{teoirred}. 
\end{proof}

Next, we show an example of a representation-finite algebra $A$ which is not a monomial algebra neither a toupie algebra. We show how Corollary \ref{irred} allows us to reduce the number of paths that we have to analyze to compute the nilpotency index of the radical of $\mbox{mod} \,A$.

\begin{ej}\label{example ra=rb}
\emph{Consider $A=kQ_A/I_A$ the algebra given by the quiver}
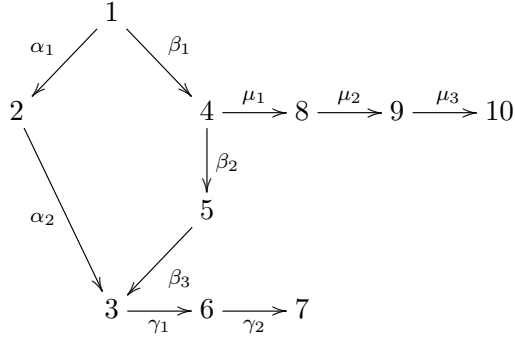
\[\xymatrix{
&1\ar[dl]_{\alpha_1}\ar[dr]^{\beta_1}&&&&\\
2\ar[ddr]_{\alpha_2}&&4\ar[d]^{\beta_2}\ar[r]^{\mu_1}&8\ar[r]^{\mu_2}&9\ar[r]^{\mu_3}&10\\
&&5\ar[dl]^{\beta_3}&&&\\
&3\ar[r]_{\gamma_1}&6\ar[r]_{\gamma_2}&7&&&
}\]
\emph{with $I_A=<\alpha_2\alpha_1-\beta_3\beta_2\beta_1, \mu_3\mu_2, \gamma_2\gamma_1>$. 
It follows from Theorem \ref{nilpopf} that the vertices $1$, $7$ and $10$ do not have to be considered in order to study the nilpotency index of the radical of $\mbox{mod} \,A$.}

\emph{On the other hand, note that $\mbox{dim}_k\mbox{Hom}_A(P_i,P_j)\leq 1$ and $\mbox{dim}_k\mbox{Hom}_A(I_i,P_j)\leq 1$ for all $i,j\in (Q_A)_0$. Then applying Corollary \ref{irred} we get that: 
\begin{itemize}
    \item Since there are irreducible morphisms $I_9\rightarrow I_8$ and from $I_6\rightarrow I_3$ then $r_8\leq r_9$ and $r_3\leq r_6$. 
    \item Since there are irreducible morphisms $P_8\rightarrow P_4$ and $I_8\rightarrow I_4$, then  $r_8=r_4$.  
    \item Since there are irreducible morphisms $P_5\rightarrow P_4$ and $I_5\rightarrow I_4$, then $r_5=r_4$.
    \item Since there are irreducible morphisms $P_3\rightarrow P_5$ and from $P_3\rightarrow P_2$, then $r_3\leq r_5$ and $r_3\leq r_2$. 
\end{itemize}
Hence, to compute the nilpotency index  of the radical of $\mbox{mod}\,A$ we only have to study the vertices $2, 9$ and $6$. Computing the respectively paths from the projective to the injective going through the simple in such vertices we obtain that $r_2=27$, $r_9=27$ and $r_4=26$. Therefore, $r_A= \mbox{max} \{r_{i} +1\}_{i=2, 4,9}=28$.}
\end{ej}

\section{Application to monomial algebras}

Consider $A$ a representation-finite algebra. The aim of this section is to extend \cite[Theorem 2.5]{CGS} to any monomial algebra, that is, monomial algebras whose Auslander-Reiten quiver may be a component without length.
\vspace{.in}

We start this section recalling the following definition.

\begin{defi}\label{involved}
Let $A \simeq kQ_A/I_A$.  We say that a vertex $x$ in $Q_A$ is involved in a zero-relation $\alpha_m \dots \alpha_1 \in I_A$ if $x=s(\alpha_i)$ for some $i=2, \dots, m$, where $s(\alpha_i)$ denotes the vertex which is the starting of the arrow $\alpha_i$.
\end{defi}

We denote by $(R_A)_0$ the set of all the vertices $u \in (Q_A)_0$ involved in zero-relations of $I_A$. 

In the next result we show that if there is a vertex $b$ in $Q_A$ not involved in a zero-relation of $I_A$,  then  $\mbox{dim}_{k}(\mbox{End}_{A}(P_b))=\mbox{dim}_{k}(\mbox{End}_{A}(I_b))  =1$.

\begin{lema} \label{vertice no involved} Let $A \simeq kQ_A/I_A$ be a finite dimensional monomial $k$-algebra over an algebraically closed field. Let $b \in (Q_A)_{0} \backslash (R_A)_0$. Then $\emph{dim}_{k}(\emph{End}_{A}(P_b))=\emph{dim}_{k}(\emph{End}_{A}(I_b)) =1$.
\end{lema}
\begin{proof} Assume that $\mbox{dim}_{k}(\mbox{End}_{A}(P_b)) \geq 2$. Then the representation of the projective corresponding to  the vertex $b$ is such that in $(P_{b})_{b}$ there are at least two linearly independent paths starting and ending in $b$.  Indeed, one of this paths is $e_{b}$ and the other is a non-zero cycle $\varrho: b \rightarrow b$  in $Q_A$. Since by hypothesis $b$ is not involved in a zero-relation then $\varrho^{n} \neq 0$ for all positive integer $n$, getting a contradiction to the fact that $A$ is a finite dimensional $k$-algebra. Therefore, if $b \notin (R_A)_0$ then $\mbox{dim}_{k}(\mbox{End}_{A}(P_b))=1$. The result follows since  $\mbox{End}_{A}(P_b)\simeq \mbox{End}_{A}(I_b)$. 
\end{proof}

Now, we are in position to prove the following result.

\begin{prop} \label{items-1} Let $A \simeq kQ_A/I_A$ be a finite dimensional monomial $k$-algebra over an algebraically closed field. Assume that $A$ is  representation-finite. Let $a$ and $b \in (Q_A)_{0}$ such that there is an arrow $\alpha$ from $a$ to $b$ in $Q_{A}$. Then the following conditions hold.
\begin{enumerate}
\item[(a)] If $a \in (R_A)_{0}$ and  $b \notin (R_A)_0$ then $r_{b} \leq r_{a}$.
\item[(b)] If $a \notin  (R_A)_0$ and  $b \in (R_A)_0$ then $r_{a} \leq r_{b}$.
\item[(c)] If $a \notin (R_A)_0$ and  $b \notin  (R_A)_0$ then $r_{a} = r_{b}$.
\end{enumerate}
\end{prop}

\begin{proof}
(a) Since  $b \notin (R_A)_0$ then $\mbox{dim}_{k}(\mbox{End}(P_b)) =1$. On the other hand, suppose that the morphism $P_b\rightarrow P_a$ is not an irreducible morphism, then there exists a non-zero path $\mu$ in $Q_A$ from $b$ to a vertex $x$ such that $\mu\alpha\in I$. Then $b\in (R_A)_0$ which is a contradiction. The result follows from Corollary \ref{irred} (a). 

\vspace{0.1in} 

(b) Follows with dual arguments than Statement (a).

\vspace{0.1in}

(c)  Since  $a, b \notin (R_A)_0$ then $\mbox{dim}_{k}(\mbox{End}_{A}(P_b)) =1$ and $\mbox{dim}_{k}(\mbox{End}_{A}(P_a)) =1$. Moreover, since $a$ and $b$ are not involved in a zero-relation we have that there are irreducible morphism from $P_b$ to $P_a$ and from $I_b$ to $I_a$. Then, the result follows from Corollary \ref{irred}(a) and (b). 
\end{proof}

In case that  $a,b\in (R_A)_0 $ then we have both possibilities  $r_a=r_b$, as in Example \ref{example ra=rb},  and $r_a < r_b$ as we show in the next example.

\begin{ej}
\emph{ Consider the algebra $A=kQ_A/I_A$ given by the following quiver 
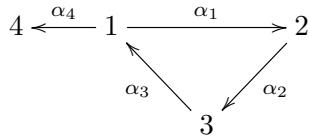
\[ \xymatrix{4&1\ar[l]_{\alpha_4}\ar[rr]^{\alpha_1}&&2\ar[dl]^{\alpha_2}\\
&&3\ar[ul]^{\alpha_3}&}\]
with $I_A=<\alpha_1\alpha_3\alpha_2\alpha_1, \alpha_4\alpha_3\alpha_2>$.  
If we consider the relation $\alpha_1\alpha_3\alpha_2\alpha_1$,  for the arrow $\alpha_2$ we have to study the vertices $2$ and $3$ and we obtain $r_2=12$ and $r_3=16$.} 
\vspace{0.1in}

\end{ej}
As an immediate consequence of  Proposition \ref{items-1}   we get Theorem B the main result of this section. 

\begin{teo} \label{vertex-nilpo} Let $A\simeq kQ_A/I_A$ be a representation-finite monomial algebra. Consider  $(R_A)_0$ to be the set of the vertices $u\in Q_A$ involved in zero-relations of $I_A$. Then the nilpotency index of $\Re(\emph{mod}\,A)$ is determined by the length of the longest non-zero path of irreducible morphisms from the projective $P_u$ to the injective $I_u$ going through the simple $S_u$, where $u \in(R_A)_0$. Precisely,  $r_A = \emph{max}\, \{ r_{u}+1 \}_{u\in (R_A)_0}$.
\end{teo}

 In the particular case that we deal with an algebra $A$ such that each vertex involved in a zero-relation appears in only one relation of $I_A$ and moreover it appears only once, then we get Theorem C. 
 
 \begin{teo}\label{rarb}
Let $A\simeq kQ_A/I_A$ be a representation-finite monomial algebra. Consider $(R_A)_0$ to be the set of the vertices $u\in Q_A$
involved in zero-relations of $I_A$.  If each vertex in $(R_A)_0$ belongs to only one zero-relation only once then the nilpotency index of $\Re(\emph{mod}\,A)$ is determined by the length of the longest non-zero  path of irreducible morphisms from the projective $P_u$ to the injective $I_u$, where $u \in(R_A)_0$. Furthermore, all the vertices $x$ related to a fixed zero-relation determine paths from $P_x$ to $I_x$ going through $S_x$ of the same length. Therefore,  it is enough to consider only one vertex in each zero-relation.
\end{teo} 

\begin{proof} Let $\alpha_n\dots \alpha_1$ be a zero-relation, $a=s(\alpha_j)$ and $b=t(\alpha_j)$ for some $2\leq j\leq n-1$.   Since each vertex in $(R_A)_0$ belongs to only one zero-relation then there is an irreducible morphism from $P_b$ to $P_a$. Indeed, if this is not the case then there exists a non-zero path $\mu:b \rightsquigarrow x$ in $Q_A$ such that $\mu\alpha_j\in I_{A}$. Therefore, $b$ is involved in two different zero-relations which is a contradiction to the hypothesis. 

 With similar arguments as above, we can affirm that there is an irreducible morphism from $I_b$ to $I_a$.


If $\mbox{dim}_k \mbox{End}_A(P_b)=1$ and $\mbox{dim}_{k}\mbox{End}_A(I_a)=1$ then by Corollary \ref{irred} we know that $r_a=r_b$. 

Now, assume that $\mbox{dim}_k \mbox{End}_A(P_b)\geq 2$.  Then there is a non-zero cycle starting in $b$ in $Q_A$.  Therefore we have the following possibilities of subquivers in $Q_A$: 

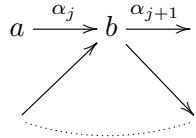
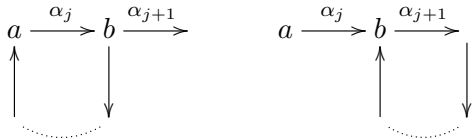
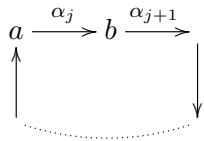
\begin{enumerate}
\item[(i)] 
\begin{displaymath}
    \begin{array}{cccc}
       \xymatrix {a\ar[r]^{\alpha_j}&b\ar[dr]\ar[r]^{\alpha_{j+1}}&\\
\ar[ur]&&\ar@{..}@/^/[ll]}&
 \end{array}
\end{displaymath}
\item[(ii)]
\begin{displaymath}
\begin{array}{cccc}
\xymatrix {a\ar[r]^{\alpha_j}&b\ar[d]\ar[r]^{\alpha_{j+1}}&\\
\ar[u]&\ar@{..}@/^/[l]&}  &  \,\,\,\,\,\,\,\,
\xymatrix {a\ar[r]^{\alpha_j}&b\ar[r]^{\alpha_{j+1}}&\ar[d]\\
&\ar[u]&\ar@{..}@/^/[l]}& 
 \end{array}
\end{displaymath}
\item[(iii)] 
\begin{displaymath}
    \begin{array}{cccc}
    \xymatrix {a\ar[r]^{\alpha_j}&b\ar[r]^{\alpha_{j+1}}&\ar[d]\\
\ar[u]&&\ar@{..}@/^/[ll]}
       \end{array}
\end{displaymath}
\end{enumerate}
\vspace{.2in}

 If there are subquivers as in Statement (ii) then we get that $b$ is involved in two different zero-relations which is a contradiction. 
 
 I  there is a subquiver as in  (i), then by  \cite[Theorem 1]{B} we get that $A$ is of infinite representation type, which is also a contradiction.

Thus the only possibility is  that $Q_A$ may have a subquiver as in Statement (iii), where the  arrows $\alpha_i$ belongs to  the cycle of the subquiver. In this case, since $\mbox{Hom}_A(P_b, P_a)$ has a basis consisting of all paths from $a$ to $b$ in $Q_A$ modulo the ideal $I_A$ we get that any non-zero morphism $\eta: P_b \rightarrow P_a$ is such that $\eta= \phi f_1$ where $\phi: P_a \rightsquigarrow P_a$  and $ f_1$ is an irreducible morphism from $P_b$ to $P_a$. Hence applying Theorem \ref{teoirred} we obtain that  $r_b \leq r_a$. 

 With similar arguments as above, if we assume that $\mbox{dim}_k \mbox{End}_A(I_b)\geq 2$ then we get that $r_a \leq r_b$. Therefore, $r_a=r_b$. 

By Theorem \ref{vertex-nilpo}, we conclude that the nilpotency index of $\Re(\mbox{mod}\,A)$ is determined by the length of the longest non-zero path of irreducible morphisms from the projective $P_u$ to the injective $I_u$, where $u \in(R_A)_0$. Furthermore, it is enough to consider only one vertex in each zero-relation. 
\end{proof}

\section{Application to toupie algebras}

In \cite{CGS}, we study the nilpotency index of the radical of a representation-finite toupie algebra with three branches and having only commutative relations. In such a case the Auslander-Reiten quiver of $\mbox{mod}\,A$ is a component with length.

In this section, we expand the study done in \cite{CGS}  to representation-finite toupie algebras where their Auslander-Reiten quivers are  components without length.
\vspace{.1in}

We start recalling the definition of toupie algebra. 
\vspace{.1in}

\begin{defi}\label{defitoupie}
A finite non linear quiver $Q$ is called  {\it toupie} if it has a unique source, a
unique sink, and for any other vertex $x$ in $Q$ there is exactly one arrow starting and one arrow ending in $x$. If $Q$ is a toupie quiver, then the algebra $A \simeq kQ_A/I_A$ is called a toupie algebra. 
\end{defi}

Throughout this section, we  consider representation-finite toupie algebras with three branches where one of them has a zero-relation and the others has a commutative relation.

 Precisely, we study the representation-finite toupie algebras
$A \simeq kQ_A/I_A$ given by the  quiver $Q_A$ as follows 
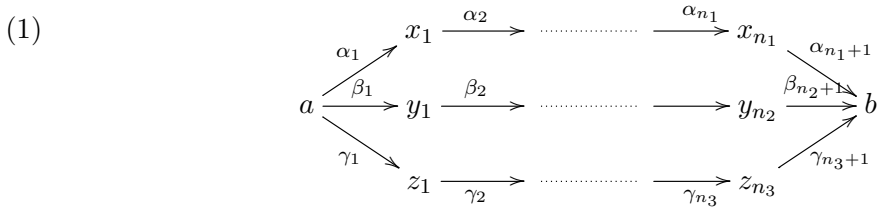
\begin{equation} \label{grafo}
\xymatrix @!0 @R=1.0cm  @C=1.5cm {&x_1\ar[r]^{\alpha_2}&\ar@{.}[r]&\ar[r]^{\alpha_{n_1}}&x_{n_1} \ar[rd] ^{\alpha_{n_1+1}}& \\
a\ar[r]^{\beta_1}\ar[ru]^{\alpha_1}\ar[rd]_{\gamma_1}&y_1\ar[r]^{\beta_2}&\ar@{.}[r]&
\ar[r]&y_{n_2}\ar[r]^{\beta_{n_2+1}}&b\\
&z_1\ar[r]_{\gamma_2}&\ar@{.}[r]&\ar[r]_{\gamma_{n_3}}&z_{n_3}\ar[ru]_{\gamma_{n_3+1}}&\\} 
\end{equation}

\noindent with $I_A=<\alpha_{n_1+1}\dots\alpha_1-\beta_{n_2+1}\dots\beta_1, \gamma_{_{t+j}}\dots \gamma_{_{j}}>$ for some $1\leq j, j+t\leq n_3$.
\vspace{.1in}

We denote by $(R_A)_0$ the vertices involved in the zero-relation. We observe that any  representation-finite toupie algebra has the property that  $\mbox{dim}_k \mbox{End}_A(P_i)=\mbox{dim}_k \mbox{End}_A(I_i)=1$, for all the vertices $i \in (Q_A)_{0}$. 




\vspace{.1in}

 The next proposition  is fundamental to determine the nilpotency index of the radical of a representation-finite toupie algebra.

\begin{prop}\label{lematoupie1}
Let $A \simeq kQ_A/I_A$ be the representation-finite toupie algebra given by the bound quiver stated in \eqref{grafo}. 
Then the following statements hold. 
\begin{enumerate}
    \item[(a)] $r_{x_i}=r_{x_j}$ for all $i,j\in \{1,\dots, n_1\}$.
    \item[(b)] $r_{y_i}=r_{y_j}$ for all $i,j\in \{1,\dots, n_2\}$.
    \item[(c)] If $z_i,z_h\in (R_A)_0$ then $r_{z_i}=r_{z_h}$. 
    \item[(d)] if $z_i,z_h\notin (R_A)_0$ then  $r_{z_i}=r_{z_h}$.
    \item[(e)] 
    {If $z_i\notin (R_A)_0$ and $z_h\in (R_A)_0$ then $r_{z_i}\leq r_{z_h}$.}
\end{enumerate}
\end{prop}
\begin{proof}
(a) Without loss of generality, we may assume that there is an arrow $x_i\rightarrow x_j$. Then we have irreducible morphisms from $P_{x_j}$ to $P_{x_i}$ and from $I_{x_j}$ to $I_{x_i}$. The result follows from the fact that  $\mbox{dim}_k \mbox{End}_A(P_i)=\mbox{dim}_k \mbox{End}_A(I_i)=1$, for all the vertices $a \in (Q_A)_{0}$ and  Corollary \ref{irred}.

Statements (b), (c) and (d) follow with similar arguments than Statement (a).

(e) Without loss of generality, we may assume that there is an arrow from $z_i$ to $z_{i+1}$ with $z_i\notin (R_A)_0$ and $z_{i+1}\in (R_A)_0$. Then, there is an irreducible morphism from $P_{z_{i+1}}$ to $P_{z_{i}}$. The result follows from  Corollary \ref{irred}  since $\mbox{End}_A(P_{z_{i+1}})\simeq k$.

\end{proof}

 The next technical result will allow us to prove that  the vertices of $Q_A$ that belong to the zero-relation are the ones that determine the nilpotency index of the radical of the module category of the considered toupie algebras.

\begin{lema}\label{three}
Let $A \simeq kQ_A/I_A$ be the representation-finite toupie algebra defined as in  \eqref{grafo}. 
Consider the indecomposable $A$-modules $M_{z_{i}}=((M{z_{i}})_c, (M{z_{i}})_\delta)_{c\in Q_0, \delta\in Q_1}$ define as follows: 

\[\begin{array}{cc}
  (M{z_{i}})_c=\left\{
  \begin{array}{ll}
    k^2 & if \; \; c=z_{i} \\
    k & \hbox{otherwise}
  \end{array}
\right., &  (M{z_{i}})_\delta=\left\{
  \begin{array}{ll}
  \left(
         \begin{array}{cc}
           1 &  0
         
         \end{array}
       \right)  & if \; \; \delta= \gamma_{i+1} \\
  \vspace{0.1in}
    \left(
         \begin{array}{c}
           0 \\
          1 \\
         \end{array}
       \right) & if \; \; \delta= \gamma_{i}\\
    \mbox{Id} & \hbox{otherwise}
  \end{array}
\right.
\end{array}\]
\noindent for $i=j, \dots, j+t-1$. Then there are non-zero cycles $\rho_{z_{i}}:M_{z_{i}}\rightsquigarrow M_{z_{i}}$,  which are compositions of $n$ non-isomorphisms where $n$ is equal to two times the number of arrows of the branch having the zero-relation.  Moreover, each cycle $\rho_{z_i}$ 
 factors through the simple $S_{z_{i}}$ and   it belongs to $\Re^{n}(M_{z_{i}}, M_{z_{i}})$.
\end{lema}
\begin{proof} Consider the $A$-module $M_{z_{i}}$. By definition, it is clear that $S_{z_{i}}$ is a summand of the socle of $M_{z_{i}}$. Then we can consider the epimorphism 
\[M_{z_{i}} \rightarrow M_{z_{i}}/S_{z_{i}}=N_1.\]
Now, we observe that  $S_{z_{{i-1}}}$ is a direct summand of the socle of $N_1$. Then  again we can consider the epimorphism $N_1\rightarrow N_1/S_{z_{{i-1}}}=N_2$. We can iterate $i$ times  this argument until we get  the simple $A$-module $S_{z_1}$ as a direct summand of the socle of  $N_{i-1}$. Therefore  we construct a chain of $i$ epimorphisms as we describe below:  
\[M_{z_{i}} \rightarrow N_1 \rightarrow N_2 \rightarrow \dots \rightarrow N_{i-1}\rightarrow N_{{i-1}}/S_{z_1}=N_{{i}}. \]

Note  that the module  $N_{{i}}=((N_{{i}})_c, (N_{{i}})_\delta )_{c\in Q_0, \delta\in Q_1}$ 
has the following representation:
\[\begin{array}{cc}
  (N_{{i}})_c=\left\{
  \begin{array}{ll}
    0 & if \; \; c\in \{z_1,\dots, z_{{i-1}}\} \\
    k & \hbox{otherwise}
  \end{array}
\right., &  (N_{{i}})_\delta=\left\{
  \begin{array}{ll}
   0  & if  \; \; \delta\in \{\gamma_1,\dots, \gamma_{i}\} \\
    \mbox{Id} & \hbox{otherwise}
  \end{array}
\right..
\end{array}\]
If we consider $L=(L_c,L_\delta)_{c\in Q_0, \delta\in Q_1}$ the $A$-module such that its representation is: 
  \[(L)_c=\left\{
  \begin{array}{ll}
    k & if  \; \; c\in \{a,x_1,\dots,x_{n_1}, y_1,\dots, y_{n_2},b\} \\
    0 & \hbox{otherwise}
  \end{array}
\right., \] 
\[(L)_\delta=\left\{
  \begin{array}{ll}
   \mbox{Id}  & if  \; \; \delta\in \{\alpha_1,\dots, \alpha_{n_1+1}, \beta_1,\dots, \beta_{n_2+1}\} \\
    0 & \hbox{otherwise}
  \end{array}
\right.,\]

\noindent then $L$ is a submodule of $N_{{i}}$ and we can consider the epimorphism $N_{{i}}\rightarrow N_{{i}}/L$. Observe that $N_{{i+1}}=N_{{i}}/L$ is an uniserial $A$-module whose top is $S_{z_{{i}}}$ and whose socle is $S_{z_{n_3}}$. Therefore, we can build the chain of $n_3-i$ epimorphisms where in each step we put the $A$-module which is a quotient of the previous one by its socle: 
\[N_{{i+1}}\rightarrow N_{{i+1}}/S_{z_{n_3}}\rightarrow \dots \rightarrow N_{n_3}\rightarrow N_{n_3}/S_{z_{{i+1}}}=S_{z_{i}}.\]
Hence we construct a chain of $n_3+1$ epimorphisms as follows 
\[M_{z_{i}} \rightarrow N_1 \rightarrow N_2 \rightarrow \dots \rightarrow N_{{i}} \rightarrow  N_{{i+1}}\rightarrow \dots \rightarrow N_{n_3}\rightarrow S_{z_{k_i}}.\]

Similarly, we can construct a chain of $n_3+1$ monomorphisms from $S_{z_{i}}$ to $M_{z_{i}}$ by considering the natural inclusions in each step. 

By construction it is clear that the cycle from $M_{z_{i}}$ to $M_{z_{i}}$ is non-zero. 

We observe that since $\mbox{End}_A(M_{z_{i}})\simeq k^2$ (we have the identity and $\rho_{z_{i}}$) and $\mbox{rad}(\mbox{End}_A(M_{z_{i}}))\simeq k$ then $\mbox{End}_A(M_{z_{i}})$ is local and, in consequence, $M_{z_{i}}$ is indecomposable. 

On the other hand, for all the intermediate $A$-modules $N_s$ use to build the chain it is not hard to see that $\mbox{End}_A(N_s)\simeq k$. Hence  $\mbox{End}_A(N_s)$ is a local algebra and in consequence, each $N_s$ is an indecomposable $A$-module. 

 Therefore, we prove that there is a non-zero chain of non isomorphisms between indecomposable $A$-modules from $M_{z_{i}}$ to $M_{z_{i}}$ that factors through the simple $S_{z_{i}}$ of length at least $n$, where $n$ is equal to two times the number of arrows of the branch having the zero-relation. Furthermore  $\rho_{z_i}\in \Re^{n}(M_{z_{i}}, M_{z_{i}})$. 
\end{proof}



\begin{obs}\label{obsmono}
\emph{Let $M_{z_i}$ and $\rho_{z_i}$ be defined as in Lemma \ref{three}. We observe that there is a monomorphism from $\phi_{z_{i}}:P_{z_{i}}\rightarrow M_{z_{i}}$ such that $\rho_{z_{i}}\phi_{z_{i}} \neq 0$. Indeed, since in the representation of the $A$-module $M_{z_i}$ the $k$-vector space $(M_{z_i})_{z_i}$ has dimension two, we may consider  $\{x,y\}$ a basis for such a vector space.}

\emph{Now, since $(M_{z_i})_{\gamma_{i}}(y)\neq 0$, we  define $\phi_{z_{i}}(e_{z_{i}})=y$. It is not hard to check that $\phi_{z_{i}}$ is a monomorphism. Moreover, since the element $y$ corresponds to an element in the top of $M_{z_i}$ we have that $\rho_{z_{i}}(y)\neq 0$ and thus we get the claim. }

\emph{With similar arguments as above we can prove  that there is an epimorphism from $\psi_{z_{i}}:M_{z_{i}}\rightarrow I_{z_{i}}$ such that $\psi_{z_{i}}\rho_{z_{i}}\neq 0$.} 
\end{obs}
\vspace{.1in}

 We illustrate Lemma \ref{three} with the following example. 
\begin{ej}\label{ej1}
\emph{Consider the toupie algebra $A=kQ_A/I_A$ given by the quiver $Q_A$}

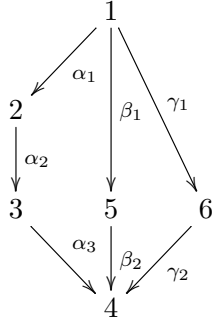
\begin{displaymath}
\xymatrix {  & 1 \ar[rdd]^{\gamma_1}  \ar[ld]^{\alpha_1} \ar[dd]^{\beta_1} & \\ 
2 \ar[d]^{\alpha_2}  &  & \\
3 \ar[rd]^{\alpha_3}  & 5 \ar[d]^{\beta_2}  & 6 \ar[ld]^{\gamma_2}\\
  & 4 & }
\end{displaymath}
\vspace{.1in}

\noindent \emph{with  $I_A=<\beta_2\beta_1-\gamma_2\gamma_1, \alpha_3\alpha_2\alpha_1>$.}

\emph{The Auslander-Reiten quiver of  $\mbox{mod} \,A$ is as follows:}
\vspace{.1in}
{
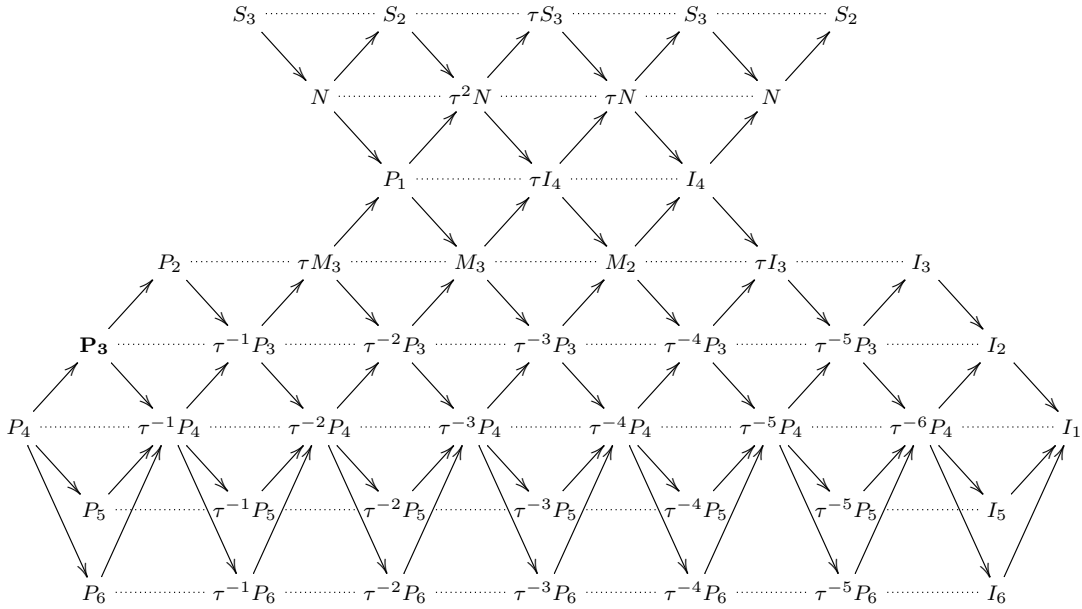
\begin{displaymath}
\def\objectstyle{\scriptstyle}
\def\labelstyle{\scriptstyle}
   \xymatrix @!0 @R=1.1cm  @C=1cm{& & & S_3 \ar[rd]\ar@{.}[rr]  & & S_2 \ar[rd]\ar@{.}[rr] & &  \tau S_3 \ar[rd]\ar@{.}[rr]  & &S_3\ar[rd]\ar@{.}[rr] &   &S_2  && && \\
   & & & & {N} \ar[rd]\ar@{.}[rr] \ar[ru] & & \tau^{2}N \ar[rd]\ar@{.}[rr] \ar[ru]&  &\tau N \ar[ru]\ar[rd] \ar@{.}[rr] & & N \ar[ru]& && && \\
  & & && & { P_1 } \ar[rd]\ar@{.}[rr] \ar[ru]& & \ar[ru]\tau I_4 \ar[rd] \ar@{.}[rr]& & I_4 \ar[rd] \ar[ru]& &  && &&\\
    & & P_2 \ar[rd] \ar@{.}[rr]& & \tau M_3 \ar@{.}[rr] \ar[rd] \ar[ru] & & M_3\ar[rd] \ar[ru] \ar@{.}[rr] & & M_2 \ar[rd] \ar@{.}[rr] \ar[ru] & & \tau I_3 \ar[rd] \ar@{.}[rr] & & I_3 \ar[rd]  & && \\
    &{\bf  P_3}\ar@{.}[rr] \ar[rd] \ar[ru] &  & { \tau^{-1} P_3 } \ar[rd]\ar[ru]\ar@{.}[rr]
   & & { \tau^{-2} P_3 } \ar[rd] \ar[ru]\ar@{.}[rr]& &  \tau^{-3} P_3 \ar[rd]\ar@{.}[rr] \ar[ru] & & \tau^{-4} P_3 \ar[rd] \ar[ru]\ar@{.}[rr] & &  \tau^{-5} P_3 \ar[rd] \ar[ru]\ar@{.}[rr] && I_2 \ar[rd] && \\
 P_4 \ar[rd] \ar[ru] \ar[rdd] \ar@{.}[rr]& & { \tau^{-1} P_4 } \ar[rd] \ar[ru] \ar[rdd]\ar@{.}[rr] &&{\tau^{-2} P_4}  \ar[rdd]\ar[rd] \ar[ru]\ar@{.}[rr]& & { \tau^{-3} P_4}   \ar[rd] \ar[ru] \ar@{.}[rr] \ar[rdd] &  &{ \tau^{-4} P_4 }  \ar[rdd]\ar[ru] \ar[rd]\ar@{.}[rr] & &  \tau^{-5} P_4 \ar[ru] \ar[rd]\ar@{.}[rr] \ar[rdd]& & {\tau^{-6} P_4}  \ar[rdd]\ar[rd] \ar[ru]\ar@{.}[rr]& & I_1& \\
  &P_5 \ar[ru]\ar@{.}[rr] & & { \tau^{-1} P_5 }  \ar[ru] \ar@{.}[rr] & & \tau^{-2} P_5 \ar[ru]\ar@{.}[rr] & &{\tau^{-3} P_5 }\ar@{.}[rr] \ar[ru] & &\tau^{-4} P_5\ar[ru] & & \tau^{-5} P_5 \ar[ru]\ar@{.}[rr]  && I_5 \ar[ru] &&\\
  & P_6 \ar[ruu]\ar@{.}[rr]&    &  { \tau^{-1} P_6 }\ar@{.}[rr]  \ar[ruu] & & \tau^{-2} P_6 \ar@{.}[rr]\ar[ruu] & & \tau^{-3} P_6\ar@{.}[rr]  \ar[ruu]   & & \tau^{-4} P_6 \ar[ruu] \ar@{.}[rr]& & \tau^{-5} P_6  \ar@{.}[rr]\ar[ruu] && I_6 \ar[ruu]&& \\}
\end{displaymath}}
\vspace{.1in}

\noindent \emph{where we identify the copies of the simple $A$-modules  $S_2$, $S_3$ and $N$. }

\emph{Observe that in this case there are two non-zero cycles} 
\emph{\[M_3={{\small\txt{\;\;\;\;\;1\\2 \\ 3\;\;\\ \\}\;\txt{$\frac{}{}$\\56\;\;\;3 \\ 4\\ \\ }}}\longrightarrow {{\small\txt{\;\;\;\;\;1\\2 \\ \;\;\\ \\}\;\txt{$\frac{}{}$\\56\;\;\;3 \\ 4\\ \\ }}} \longrightarrow {\txt{1\;\;\\56\;\;3 \\ 4\\ \\ }} \longrightarrow 3 \longrightarrow {\txt{2\\3 }}\longrightarrow {{\small\txt{\;\;\;1\\2 \\ 3\;\;\\ \\}\;\txt{$\frac{}{}$\\56 \\ 4\\ \\ }}} \longrightarrow M_3 \]}
\noindent \emph{and}
\emph{\[M_2= {\txt{\;\;\;1\;\;\;\;\;\;\;2\\2\;\;56\;\;\;3 \\ 4\\ \\ }}\longrightarrow {\txt{\;\;\;1\;\;\;\;\;\;\;2\\56\;\;\;3 \\ 4\\ \\ }} \longrightarrow  {\txt{2\\3 }} \longrightarrow 2 \longrightarrow {\txt{\;\;\;1\;\;\;\;\;\;\;\\2\;\;56\;\;\; \\ 4\\ \\ }}\longrightarrow {{\small\txt{\;\;\;\;\;1\\2 \\ \;\;\\ \\}\;\txt{$\frac{}{}$\\56\;\;\;3 \\ 4\\ \\ }}} \longrightarrow M_2\]}

\noindent \emph{where we denote each indecomposable $A$-module by their composition factors. Moreover, observe that the arrows represent irreducible morphisms. Therefore, both cycles are of the same length.
Moreover,  since both paths are sectional then by \cite[Appendix]{IT} we know that they belong to $\Re^{6} \backslash \Re^{7}$, that is, they are of length six.} 
\end{ej}

\begin{obs} \label{tecnico}
\emph{Let $A$ be a finite dimensional $k$-algebra.  Consider $M$ and $N$ indecomposable  $A$-modules such that $\dim_k (\mbox{Hom}_{A}(M,N))=1$.  Assume that  there is an irreducible morphism $f$ from $M$ to $N$. Then $f \in \Re(M,N) \backslash \Re^{2}(M,N)$. Then any other morphism $g:M \rightarrow N$ is such that $g =\alpha f$ with $\alpha \in k$. Clearly $\Re^{2}(M,N) =0$, since  otherwise $\dim_k (\mbox{Hom}_{A}(M,N)) \geq 2$.}
\end{obs}

\begin{lema}\label{lematecnico}
Let $X$ and $Y$ be indecoposable $A$-modules  such that $\emph{dim}_k\emph{Hom}_A(X,Y)=1$ and there is an irreducible morphism $f:X\rightarrow Y$. Let $\varphi:Y\rightarrow Z$ be a non-zero path of irreducible morphisms between indecomposable modules such that $\varphi \in \Re^m(Y,Z)\backslash \Re^{m+1}(Y,Z)$. If $\varphi f \in \Re^{m+2}(X,Z)$ then there exists a morphism $\overline{\varphi}\in \Re^{m+1}(Y,Z)$. Moreover $\overline{\varphi}f\neq 0$. 
\end{lema}
\begin{proof}
Let $\varphi:Y=X_1 \stackrel{h_1}\rightarrow X_2 \stackrel{h_2}\rightarrow X_3 \rightarrow \dots \rightarrow X_m \stackrel{h_m}\rightarrow X_{m+1}=Z$ be a path of irreducible morphisms such that $\varphi \in \Re^m(Y,Z)\backslash \Re^{m+1}(Y,Z)$. Assume that $\varphi f \in \Re^{m+2}(X,Z)$. Then by Theorem \ref{explotan}, there is a  path $X=X_0\stackrel{f_0}\rightarrow Y \stackrel{f_1}\rightarrow X_2 \stackrel{f_2}\rightarrow X_3 \rightarrow \dots \rightarrow X_m \stackrel{f_m}\rightarrow X_{m+1}=Z$ of irreducible morphisms with zero-composition and a morphism $\epsilon= \epsilon_{_{m}} ... \epsilon_{_0} \neq 0$ where $\epsilon_{_i} = f_{_i}$  or $\epsilon{_i} \in \Re^{2}(X_i, X_{i+1})$. Observe that at least for one $i$ we have that $\epsilon¨{_i} \in \Re^{2}(X_i, X_{i+1})$, otherwise $\epsilon=0$. 

By Remark \ref{tecnico} we know that $\epsilon_{_0}=cf$, with $c\in k-\{0\}$, because $\mbox{dim}_{k}\mbox{Hom}_{A}(X,Y)=1$ and since by hypothesis there is an irreducible morphism from $X$ to $Y$. Therefore, we have  a morphism $\overline{\varphi}=\epsilon_{_{m}} ... \epsilon_{_1} \in \Re^{m+1}(Y, Z)$ and clearly $\overline{\varphi}f\neq 0$. 
\end{proof}

 Now, we are in position to prove Theorem D, the main result of this section.

\begin{teo}\label{toupie}
Let $A \simeq kQ_A/I_A$ be the representation-finite toupie algebra given by the bound quiver  defined as in $(\ref{grafo})$.
Then the nilpotency index of $\Re(\emph{mod}\,A)$ is determined by the length of the   non-zero path of irreducible morphisms from the projective $P_u$ to the injective $I_u$ going through the simple $S_u$, where $u$ is a vertex involved in the zero-relation.
Furthermore, all the vertices $x$ related to the zero-relation  determine paths from $P_x$ to $I_x$ going through $S_x$ of the same length. Therefore, it is enough to consider only one vertex in the zero-relation.
\end{teo}

\begin{proof} It follows from Proposition \ref{lematoupie1} that it is enough to study $r_{x_{n_1}}$,   $r_{y_{n_2}}$ and $r_{z_{{i}}}$ for some $j\leq i \leq j+t-1$. We only prove that $r_{y_{n_2}} \leq r_{z_{{i}}}$, since the fact that $r_{x_{n_1}} \leq r_{z_{{i}}}$ follows with similar arguments.

Consider the $A$-module $M_{z_{i}}$ and the cycle $\rho_{z_i}\in \mbox{End}_k(M_{z_{i}})$  defined in Lemma \ref{three}. Note that $\mbox{dim}_{k}(\mbox{Hom}_{A}(P_b, M_{z_{i}}))=1$ and also  that $\mbox{dim}_{k}(\mbox{Hom}_{A}( M_{z_{i}},I_a))=1$.  

It is clear  from the shape of the quiver $Q_A$ that there are irreducible morphisms from $P_{z_{c+1}}$ to $P_{z_{c}}$ for $i\leq c \leq n_3-1$, and from $I_{z_{r+1}}$ to $I_{z_{r}}$ for $1 \leq r\leq i-1$. 

Consider  $\phi_{z_i}:P_{z_{i}}\rightarrow M_{z_{i}}$ the natural inclusion. By Remark \ref{obsmono}, we have that $\rho_{z_i}\phi_{z_i}\neq 0$. 
If $\phi_{z_i}=\sum\limits_{l=1}^L \phi_l$, where each $\phi_l$ is a  path of irreducible morphisms between indecomposable $A$-modules then there exists $l$ such that $\rho_{z_i}\phi_l\neq 0$. Assume that $\phi_l\in \Re^{s}(P_{z_{i}}, M_{z_{i}}) \backslash \Re^{s+1}(P_{z_{i}},M_{z_{i}})$. Since $\rho_{z_i}$ is a non-zero morphism, $\rho_{z_i}\phi_l\in \Re^{s+1}(P_{z_{i}},M_{z_{i}})$ and thus $\phi_l, \rho_{z_i}\phi_l$ are linear independent morphisms as $k$-vector spaces.  We know that $\mbox{dim}_k\mbox{Hom}_A(P_{z_{i}},M_{z_{i}})=2$ then  $\{\phi_l, \rho_{z_i}\phi_l\}$ is a basis for $\mbox{Hom}_A(P_{z_{i}},M_{z_{i}})$. In particular, since $\phi_{z_i}\in \mbox{Hom}_A(P_{z_{i}},M_{z_{i}})$, then there exist $a_1, a_2\in k$ such that $\phi_{z_i}=a_1 \phi_l+a_2\rho_{z_i}\phi_l$.


We claim that $\phi_l$ is a monomorphism. In fact, let $h$ be a morphism such that $\phi_lh=0$. Then $\phi_{z_i} h=a_1 \phi_lh+a_2 \rho\phi_lh=0$. Since $\phi_{z_i}$ is a monomorphism we get that $h=0$ and thus, $\phi_l$ is a monomorphism. 

Let $f$ be the composition of the $i$ irreducible monomorphism between projective $A$-modules from $P_{z_{n_3}}$ to $P_{z_{i}}$. Then $\phi_{l}f\neq 0$ since they are monomorphisms. 

We claim that $\phi_{l}f\in \Re^{s+i}(P_b, M_{z_{i}}) \backslash \Re^{s+i+1}(P_b,M_{z_{i}})$. 
Indeed, otherwise if $\phi_{l}f\in \Re^{s+i+1}(P_b, M_{z_{i}})$  applying Lemma \ref{lematecnico} $i$ times, we get that there exists $\overline{\phi_l}\in \Re^{s+1}(P_{z_{i}}, M_{z_{i}})$ such that  $\overline{\phi_l}f\neq 0$. 
Now, since $\overline{\phi_l}\in \mbox{Hom}_A(P_{z_{i}}, M_{z_{i}})$ then $\overline{\phi_l}=b_1\phi_l+b_2\rho_{z_i}\phi_l$, with $b_1,b_2\in k$. Moreover, since $\overline{\phi_l}\in \Re^{s+1}(P_{z_{i}}, M_{z_{i}})$, then  $b_1$ is zero. Hence $\overline{\phi_l}f=b_2\rho_{z_i}\phi_lf=0$ because $\rho_{z_i}$ factors through $S_{z_{i}}$ and $\mbox{Hom}_A(P_b,S_{z_{i}})=0$. Thus we get a contradiction. Therefore, $\phi_lf\in \Re^{s+i}(P_b, M_{z_{k_i}}) \backslash \Re^{s+i+1}(P_b,M_{z_{k_i}})$. 

On the other hand, consider $g: P_b\rightarrow P_{y_{n_2}}$ an irreducible morphism, and $\varphi:P_{y_{n_2}}\rightarrow M_{z_{i}}$ the natural inclusion. Since $\mbox{dim}_k\mbox{Hom}_A(P_{y_{n_2}}, M_{z_{i}})=1$ there is a path of irreducible morphism between indecomposable $A$-modules of maximal length which is a monomorphism. Let  denote it by $\phi':P_{y_{n_2}}\rightarrow M_{z_{k_i}}$ with $\phi'\in \Re^{n}(P_{y_{n_2}}, M_{z_{k_i}}) \backslash \Re^{n+1}(P_{y_{n_2}}, M_{z_{k_i}})$.

We claim that $n=s+i-1$. Indeed,  since $\mbox{dim}_k\mbox{Hom}_A(P_{b}, M_{z_{i}})=1$ we have that $\phi' g = d_1 \phi_lf\in \Re^{s+i}(P_b, M_{z_{k_i}}) \backslash \Re^{s+i+1}(P_b,M_{z_{k_i}})$, with $d_1\in k$. Then if $n< s+i-1$, applying Lemma \ref{lematecnico}, we have that there exists $\overline{\phi'}\in \Re^{n+1}(P_{y_{n_2}}, M_{z_{k_i}}) $.  Since $\mbox{dim}_k\mbox{Hom}_A(P_{y_{n_2}}, M_{z_{i}})=1$, then $\overline{\phi'} = d_2 \phi'$, with $d_2\in k$, which is a contradiction because $\phi'\in \Re^{n}(P_{y_{n_2}}, M_{z_{k_i}}) \backslash \Re^{n+1}(P_{y_{n_2}}, M_{z_{k_i}})$. Hence $n=s+i-1$ and $\phi'\in \Re^{s+i-1}(P_{y_{n_2}}, M_{z_{i}}) \backslash \Re^{s+i}(P_{y_{n_2}},M_{z_{i}})$.  

Summarizing, we have the following diagram:
$$\xymatrix @!0 @R=1.0cm  @C=1.3cm {&  P_{z_{n_3}} \ar[r]&\ar@{.}[r]&\ar[r]&
P_{{z_{i}}}\ar[rrd]^{\phi_l}& & \\
P_b\ar[ru]\ar[rrd]&&& & & &M_{{z_{i}}}\\
& &P_{y_{n_2}}\ar[rrrru]_{\phi'}&&& \\} $$
with $\phi_l\in \Re^{s}(P_{z_{i}}, M_{z_{i}}) \backslash \Re^{s+1}(P_{z_{i}},M_{z_{i}})$ and $\phi'\in \Re^{s+i-1}(P_{y_{n_2}}, M_{z_{i}}) \backslash \Re^{s+i}(P_{y_{n_2}},M_{z_{i}})$.

Similarly,  consider  $\psi=\sum\limits_{r=1}^R \psi_r$ from  $M_{z_{i}}$ to $I_{z_{i}}$, the natural projection, where $\psi_r$ is a  path of irreducible morphisms between indecomposable $A$-modules. 

With dual arguments as before we have that there exists $\psi_r:M_{z_{i}}\rightarrow I_{z_{i}}$ such that $\psi_r\rho_{z_i}\neq 0$, $\psi_r\in \Re^{m}( M_{z_{i}}, I_{z_{i}}) \backslash \Re^{m+1}(M_{z_{i}}, I_{z_{i}})$ is an epimorphism. Moreover, if $\overline{f}$ is the composition of $n_3-i+1$ irreducible epimorphisms from $I_{z_{i}}$ to $I_a$ then $\overline{f}\psi_r\in \Re^{m+n_3-i+1}( M_{z_{i}}, I_{a}) \backslash \Re^{m+n_3-i}(M_{z_{i}}, I_{a})$. 

On the other hand, let $\psi'$ be a path of irreducible morphisms between indecomposable $A$-modules of maximal length from $M_{z_i}$ to $I_{y_{n_2}}$ and $\overline{g}$ the composition of the $n_2$  irreducible epimorphisms between $I_{y_{n_2}}$ to $I_a$. Since, $\mbox{dim}_k\mbox{Hom}_A(M_{z_i}, I_a)=1$ we have that $\overline{g}\psi'=c'\overline{f}\psi_r\in \Re^{m+n_3-i+1}( M_{z_{i}}, I_{a}) \backslash \Re^{m+n_3-i}(M_{z_{i}}, I_{a})$, with $c'\in k$. 

Again, with dual arguments as before we can prove that the morphism  $\psi'$ is such that it belongs to $\Re^{m+n_3-i+1-n_2}( M_{z_{i}}, I_{y_{n_2}}) \backslash \Re^{m+n_3-i-n_2}(M_{z_{i}}, I_{y_{n_2}})$.  

Therefore   we have also the following diagram:  
$$\xymatrix @!0 @R=1.0cm  @C=1.3cm {& &I_{{z_{k_i}}}\ar[r]&\ar@{.}[r]&
\ar[r]&I_{{z_{1}}}\ar[rd]& \\
M_{{z_{k_i}}}\ar[rru]^{\psi}\ar[rrd]_{\psi'}&&& & & &I_{a}\\
& &I_{y_{n_2}}\ar[r]&\ar@{.}[r]&\ar[r]&I_{y_{1}}\ar[ur]& \\} $$
with $\psi_r\in \Re^{m}( M_{z_{i}}, I_{z_{i}}) \backslash \Re^{m+1}(M_{z_{i}}, I_{z_{i}})$ and $\psi'\in \Re^{m+n_3-i+1-n_2}( M_{z_{i}}, I_{y_{n_2}}) \backslash \Re^{m+n_3-i-n_2}(M_{z_{i}}, I_{y_{n_2}})$.  

Since $\mbox{dim}_k\mbox{Hom}_A(P_{y_{n_2}}, M_{z_{i}})=1$ and $\mbox{dim}_k\mbox{Hom}_A(M_{z_{i}}, I_{y_{n_2}})=1$ we get that $\psi'\phi'$ is non-zero. Observe that $\psi'\phi'\in \Re^{m+s+n_3-n_2}(P_{y_{n_2}}, I_{y_{n_2}})\backslash \Re^{m+s+n_3-n_2+1}(P_{y_{n_2}}, I_{y_{n_2}})$ otherwise applying Theorem \ref{explotan} we get a contradicion to the fact that $\phi'$ and $\psi'$ are of maximal lenght. Moreover, since $\mbox{Hom}_A(P_{y_{n_2}},I_{y_{n_2}})\simeq k$, $\psi'\phi'$ factors through $S_{y_{n_2}}$. Then $r_{y_{n_2}}=m+s+n_3-n_2$.

Finally,  by construction, we get that  $\psi_r\rho_{z_i}\phi_l$ is a non-zero morphism from $P_{z_{i}}$ to $I_{z_i}$ that factors through $S_{z_{i}}$. Moreover,  $\psi_r\rho_{z_i}\phi_l \in \Re^{s+m+2(n_3+1)}$. Therefore, we conclude that $r_{z_{i}}\geq s+m+2(n_3+1)\geq m+s+n_3-n_2=r_{y_{n_2}}$, proving the result.  
\end{proof}

We end this section showing an example that when we have more than one zero-relation in a branch of a representation-finite toupie algebra  then Theorem D does not hold. 

\begin{ej}\label{ej1}
\emph{Consider the toupie algebra $A=kQ_A/I_A$ given by the quiver $Q_A$}
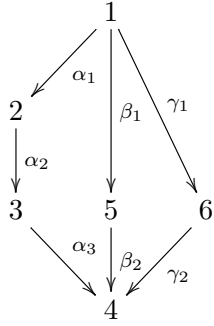
\begin{displaymath}
\xymatrix {  & 1 \ar[rdd]^{\gamma_1}  \ar[ld]^{\alpha_1} \ar[dd]^{\beta_1} & \\ 
2 \ar[d]^{\alpha_2}  &  & \\
3 \ar[rd]^{\alpha_3}  & 5 \ar[d]^{\beta_2}  & 6 \ar[ld]^{\gamma_2}\\
  & 4 & }
\end{displaymath}
\noindent \emph{with $I_A=<\beta_2\beta_1-\gamma_2\gamma_1, \alpha_2\alpha_1, \beta_2\beta_1>$.}

\emph{The Auslander-Reiten quiver of  $\mbox{mod} \,A$ is as follows:}
\vspace{.1in}
{
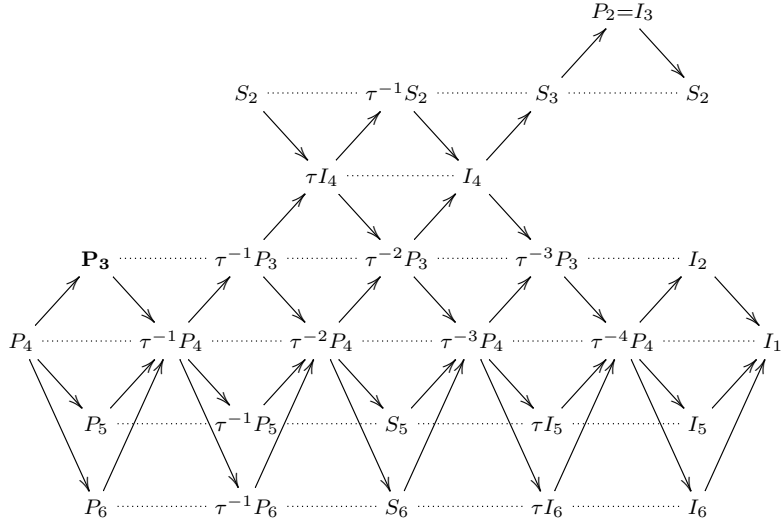
\begin{displaymath}
\def\objectstyle{\scriptstyle}
\def\labelstyle{\scriptstyle}
   \xymatrix @!0 @R=1.1cm  @C=1cm{& & & & & & &  &P_2=I_3\ar[rd] & &    \\
  & & &S_2 \ar[rd]\ar@{.}[rr] & & { \tau^{-1} S_2 } \ar[rd]\ar@{.}[rr] & & S_3\ar[ru] \ar@{.}[rr]& & S_2&    \\
    & & & & { \tau I_4 }\ar@{.}[rr] \ar[rd] \ar[ru] & & I_4\ar[rd] \ar[ru] & && &   \\
    &{\bf  P_3}\ar@{.}[rr] \ar[rd] &  & { \tau^{-1} P_3 } \ar[rd]\ar[ru]\ar@{.}[rr]
   & & { \tau^{-2} P_3 } \ar[rd] \ar[ru]\ar@{.}[rr]& &{ \tau^{-3} P_3 }\ar[rd]\ar@{.}[rr]& & I_2 \ar[rd] &  \\
 P_4 \ar[rd] \ar[ru] \ar[rdd] \ar@{.}[rr]& & { \tau^{-1} P_4 } \ar[rd] \ar[ru] \ar[rdd]\ar@{.}[rr] &&{ \tau^{-2} P_4 }  \ar[rdd]\ar[rd] \ar[ru]\ar@{.}[rr]& & { \tau^{-3} P_4 }   \ar[rd] \ar[ru] \ar@{.}[rr] \ar[rdd] &  &{ \tau^{-4} P_4 }  \ar[rdd]\ar[ru] \ar[rd]\ar@{.}[rr] & & I_1 \\
  &P_5 \ar[ru]\ar@{.}[rr] & & { \tau^{-1} P_5 }  \ar[ru] \ar@{.}[rr] & & S_5 \ar[ru]\ar@{.}[rr] & &{ \tau I_5 }\ar@{.}[rr] \ar[ru] & & I_5\ar[ru] & \\
  & P_6 \ar[ruu]\ar@{.}[rr]&    &  { \tau^{-1} P_6 }\ar@{.}[rr]  \ar[ruu] & &S_6 \ar@{.}[rr]\ar[ruu] & & { \tau I_6 }\ar@{.}[rr]  \ar[ruu]   & & I_6 \ar[ruu]& }
\end{displaymath}}
\vspace{.1in}

\noindent \emph{where we identify the two copies of the simple $S_2$. }

\emph{We observe that the path from the projective to the injective going through the simple corresponding to the vertices $5$ and $6$ are the ones that determine the nilpotency index of the radical of the toupie algebra. Moreover, we also observe that such vertices belong to the branches involved in a commutative relation.}
\end{ej}

\end{document}